\documentclass[12pt]{amsart}

\usepackage{a4wide}
\usepackage{amssymb,upgreek}
\usepackage{amsmath}
\usepackage{amsthm}
\usepackage{mathrsfs,fancybox,mdframed}
\usepackage{bm,euscript,csquotes,comment,appendix,multirow}
\usepackage{upgreek}
\usepackage[curve]{xypic}
\usepackage[usenames,dvipsnames]{xcolor}
\usepackage{enotez}
\usepackage{calligra,yfonts}
\DeclareMathOperator{\cHom}{\mathscr{H}\text{\kern -3pt {\calligra\large om}}\,}
\usepackage{multirow}
\usepackage{tikz-cd}
\usepackage{enumerate}

\usepackage{euscript}
\usepackage{hyperref}

\hypersetup{colorlinks,  citecolor=ForestGreen,  linkcolor=blue, urlcolor=black}

\def\barL{\overline{L}}
\def\barM{\overline{M}}

\def\Stab{{\rm{Stab}}}
\def\SR{{\rm{SR}}}
\def\uz{\underline{z}}

\def\la{\langle}
\def\ra{\rangle}
\def\aff{{\rm{aff}}}
\def\Waff{W_{\rm\aff}}
\def\Saff{S_{\rm\aff}}

\def\R{\mathbb{R}}
\def\C{\mathbb{C}}

\def\bfR{\mathbf{R}}

\def\q{{ \mathbf q}}
\theoremstyle{plain}
\def\k{{{k}}}

\def\pt{{\rm{pt}}}
\def\N{\mathbb{N}}
\def\Z{\mathbb{Z}}
\def\cx{{\check x}}
\def\cX{{\check X}}
\def\cY{{\check Y}}
\def\cQ{{\check Q}}
\def\cPhi{{\check \Phi}}
\def\cPi{{\check \Pi}}

\def\Q{\mathbb{Q}}

\def\cG{{\check G}}
\def\cT{{\check T}}
\def\cB{{\check B}}
\def\cmfb{{\check {\mathfrak b}}}
\def\cmfg{{\check {\mathfrak g}}}
\def\gr{{\rm{gr}}}

\newcommand{\id}{\operatorname{id}}

\newcommand{\Rees}{\textnormal{Rees}}

\def\St{{\rm{St}}}

\def\cl{{\check{\lambda}}}
\def\l{\lambda}
\def\cmu{{\check{\mu}}}

\def\cla{{\check{\lambda}}}
\def\cLa{{\check{\Lambda}}}
\def\cA{{\mathcal{A}}}
\def\cC{{\mathcal{C}}}
\def\cH{{\mathcal{H}}}
\def\cO{{\mathcal{O}}}

\def\cZ{{\mathcal{Z}}}

\def\comega{{\check{\omega}}}

\def\calpha{\check{\alpha}}

\def\r{{R}}

\newtheorem{theorem}{Theorem}[section]
\newtheorem{corollary}[theorem]{Corollary}
\newtheorem{lemma}[theorem]{Lemma}
\newtheorem{proposition}[theorem]{Proposition}

\newtheorem*{theorem*}{Theorem}
\newtheorem*{proposition*}{Proposition}

\theoremstyle{remark}

\newtheorem{remark}[theorem]{Remark}

\theoremstyle{definition}

\makeatletter
\def\@tocline#1#2#3#4#5#6#7{\relax
  \ifnum #1>\c@tocdepth 
  \else
    \par \addpenalty\@secpenalty\addvspace{#2}%
    \begingroup \hyphenpenalty\@M
    \@ifempty{#4}{%
      \@tempdima\csname r@tocindent\number#1\endcsname\relax
    }{%
      \@tempdima#4\relax
    }%
    \parindent\z@ \leftskip#3\relax \advance\leftskip\@tempdima\relax
    \rightskip\@pnumwidth plus4em \parfillskip-\@pnumwidth
    #5\leavevmode\hskip-\@tempdima
      \ifcase #1
       \or\or \hskip 1em \or \hskip 2em \else \hskip 3em \fi%
      #6\nobreak\relax
    \dotfill\hbox to\@pnumwidth{\@tocpagenum{#7}}\par
    \nobreak
    \endgroup
  \fi}
\makeatother

\title{On the center of the generic affine Hecke algebra}
\author{Sabin Cautis, Rachel Ollivier}
\address{University of British Columbia,  1984 Mathematics Road, Vancouver, BC V6T 1Z2, Canada}
\email{cautis@math.ubc.ca}
\email{ollivier@math.ubc.ca}

\begin{document}

\begin{abstract}
We identify the center of the generic affine Hecke algebra $\cH_\q$ corresponding to some root datum with the semigroup algebra $\C[\q][\cX^+]$ of the dominant chamber of its coweight lattice. This is done by first identifying a maximal commutative subalgebra $\cA_\q \subset \cH_\q$ with the Rees algebra associated to the semigroup algebra of the coweight lattice for the filtration induced by the length function. We explain how this subalgebra $\cA_\q$ can also be identified with (quantum) cohomology of the toric variety given by the fan corresponding to the coroot lattice (a Hessenberg variety). 
\end{abstract}

\maketitle

\setcounter{tocdepth}{1}

\tableofcontents

\section{Introduction}

The generic affine Hecke algebra $\cH_\q$ associated to a root datum $(X,\cX, \Phi, \cPhi)$ interpolates between the usual affine Hecke algebra $\cH_{\q^{\pm 1}}$ and the Hecke algebra $\cH_0$ with parameter $0$ which appears naturally in the mod $p$  representation theory of $p$-adic reductive groups. Denote by $G$ the group associated to the root datum and $\cG$ its Langlands dual. A celebrated result of Kazhdan and Lusztig \cite{KL} (extended by Ginzburg, see \cite{CG}, and Lusztig \cite{Lu2}) states that if $\cG$ has simply connected derived subgroup then the algebra $\cH_{\q^{\pm 1}}$ has a geometric description as the K-theory of the Steinberg stack 
$$\St_{\cG} := \cmfb/(\cB \times \C^\times) \times_{\cmfg/(\cG \times \C^\times)} \cmfb/(\cB \times \C^\times)$$
where $\cT \subset \cB \subset \cG$ is a Borel (and the torus) and $\cmfb \subset \cmfg$ their Lie algebras. The $\C^\times$ acts with weight two on $\cmfg$ and $K(\pt/\C^\times) \cong \C[\q^{\pm}]$ identifies the invertible parameter $\q^{\pm 1} \in \cH_{\q^{\pm 1}}$. We will write $\cG_\gr = \cG \times \C^\times$ and similarly $\cB_\gr, \cT_\gr$.

Their result also identified the center $\cZ_{\q^{\pm 1}} \subset \cH_{\q^{\pm 1}}$ with $K(\pt/\cG_\gr) = \C[\cG_\gr]^{\cG_\gr}$. This in turn can be identified with $K(\pt/\cT_\gr)^{W_0}$ where $W_0$ is the finite Weyl group. Geometrically $K(\pt/\cT_\gr) \cong K(\cmfb/\cB_\gr)$ can be realized as a subalgebra of $\cH_{\q^{\pm 1}}$ by the diagonal embedding 
$$\Delta: \cmfb/\cB_\gr \to \St_{\cG}.$$
Thus, if we denote $\cA_{\q^{\pm 1}} := K(\pt/\cT_\gr) \cong \C[\q^{\pm 1}][\cX]$, then $\cZ_{\q^{\pm 1}} \cong \cA_{\q^{\pm 1}}^{W_0}$. Finally, by \cite[Thm. 6.1]{Stein1}, it follows that $\C[\cX]^{W_0} \cong \C[\cX^+]$ where $\cX^+ \subset \cX$ consists of dominant coweights. In conclusion, if $\cG$ has simply connected derived subgroup, this implies that $\cZ_{\q^{\pm 1}} \cong \C[\q^{\pm 1}][\cX^+]$. 

Similarly, it was shown in \cite{Ollcompa}  using \cite{Vig1}, that the center $\cZ_0$ of $\cH_0$ is isomorphic to $\cA_0^{W_0} \cong \C[\cX^+]$ for a commutative subalgebra $\cA_0 \subset \cZ_0$. These descriptions of $\cZ_{\q^{\pm}}$ and $\cZ_0$ extend naturally to $\cZ_\q \subset \cH_\q$ as follows. By work of Vign\'eras \cite{VigI,VigII} there exists a commutative subalgebra $\cA_\q \subset \cH_\q$ whose base change recovers $\cA_{\q^{\pm 1}} \subset \cH_{\q^{\pm 1}}$ and $\cA_0 \subset \cH_0$. Moreover, $\cZ_\q \cong \cA_\q^{W_0}$. In this paper we identify $\cA_\q$ with the semigroup algebra $\C[\cC(\cX)]$ where $\cC(\cX) \subset \cX \oplus \Z$ is a cone determined by the length function $\ell$ on $\cX$. Equivalently, $\C[\cC(\cX)]$ is the Rees algebra $\Rees(\C[\cX])$ with respect to the filtration induced by $\ell$. Under this identification $\q$ corresponds to the Rees parameter. Our main result (Theorem \ref{thm:main}) implies that 
\begin{equation}\label{eq:main}
\C[\cC(\cX)]^{W_0} \cong \C[\q][\cX^+].
\end{equation}
As a consequence we find that $\cZ_\q \cong \C[\q][\cX^+]$ without the need to impose further conditions on $\cG$. This implies and interpolates between the classical isomorphism $\cZ_{\q^{\pm 1}} \cong \C[\q^{\pm 1}][\cX^+]$ and the $\q=0$ version $\cZ_0 \cong \C[\cX^+]$. By results in \cite{CO1}, it also implies that, if $X/\Z[\Phi]$ is free, the algebra $\cH_\q$ is Frobenius over its center with a Nakayama automorphism which we can identify explicitly (Corollary \ref{cor:frob}).   

In a sense, Theorem \ref{thm:main} can be interpreted as interpolating between the Chevalley-Shephard-Todd theorem (which applies to polynomials) and the Pittie-Steinberg theorem \cite[Thm. 1.1]{Stein2} (which applies to Laurent polynomials). This will be made precise in \cite{CO2} where we use the ideas from this paper to obtain a version of the Chevalley-Shephard-Todd theorem that applies to more general toric varieties. 

In Section \ref{sec:toric} we note that the algebras $\cA_0$ and, more generally, $\cA_\q$ have a geometric interpretation in terms of the cohomology of toric varieties. More precisely, to a root system one can associate a fan $\Sigma$ corresponding to its Weyl chambers. The associated toric variety $X_\Sigma$ has been studied in a variety of contexts \cite{CP,Pr,DL,BB1}. For example, in type A, it is the Losev-Manin moduli space \cite{LM} which is closely related to the moduli space of pointed genus $0$ curves. The variety $X_\Sigma$ can also be identified with a regular semisimple Hessenberg variety \cite{dMPS} and the Weyl group action on its cohomology with the ``dot representation'' \cite{Tym}. This is related the Stanley-Stembridge conjecture \cite{SW,BC} (which seems to have recently been proven using a probability argument \cite{Hik}).

We show that $\cA_0$ is isomorphic to the equivariant cohomology $H_T^*(X_\Sigma,\C)$ and that $\cA_\q$ recovers the quantum cohomology with respect to the K\"ahler class corresponding to $\ell$ (Propositions \ref{prop:coh1} and \ref{prop:coh2}). As one would expect, the cone $\cC(\cX)$ has a natural generalization $\cC^\varphi(\cX)$ corresponding to any (strictly) convex $\Sigma$-piecewise linear function $\varphi$ and $\C[\cC^\varphi(\cX)]$ turns out to be the (equivariant) quantum cohomology with respect to the K\"ahler class $\varphi$. 
 
\subsection{Acknowledgements}
We thank Kalle Karu for discussing some concepts in the theory of toric varieties and suggesting that the Stanley-Reisner ideal might play a role. S.C. was supported by NSERC Discovery grant RGPIN-2019-03961 and R.O. by NSERC Discovery grant RGPIN-2019-03963.

\section{Notation and setup}

\subsection{Based root systems} \label{sec:rootsystem}
We will work over a fixed base ring $\r$. Consider a (reduced) based root system $(X, \cX, \Phi, \cPhi, \Pi,\cPi)$  (cf. \cite[1.1]{Lu1}) where $X$ and $\cX$ are free abelian groups of finite rank equipped with a perfect pairing $\la -,- \ra : X \times \cX \to \Z$. The finite sets $\Phi \subset X$ and $\cPhi \subset \cX$ are the sets of roots and coroots. There is a bijection $\alpha \leftrightarrow \check\alpha$ such that $\la \alpha ,\check\alpha \:\ra = 2$. For  every $\alpha\in \Phi$, the reflections 
 $s_\alpha: X \to X ,\: x \mapsto x- \la x, \calpha \ra \alpha$ and $ s_{\calpha}: \cX \to \cX , \: \cx \mapsto \cx- \la \alpha, \cx \ra \calpha$, preserve $\Phi$ and $\Phi^\vee$ respectively.  The base $\Pi \subset \Phi$ consists of simple roots and defines the sets $\Phi^+$ and $\Phi^-$ of positive and negative roots.  We define $\cPi=\{\calpha, \: \alpha\in \Pi\}$. 
 
We denote by $W_0$ the finite Weyl group, namely the subgroup of ${\rm GL}(\cX)$ generated by $\{s_{\calpha}\}_{\alpha\in \Phi}$.  It can be naturally identified with the subgroup of ${\rm GL}( X)$ generated by $\{s_{ \alpha}\}_{\alpha\in \Phi}$ (\cite[1.1]{Lu1}). Then $(W_0, S_0)$ is a (finite) Coxeter system, where $S_0:=\{s_{\calpha}\}_{\alpha \in \Pi} $.
 
Denote by $Q = \Z[\Phi]$ the root lattice and by $\cQ= \Z[\cPhi]$ to  the coroot lattice.  We define the $\R$-linear vector spaces $X_\R := X \otimes_{\Z} \R$  and $\cX_\R := \cX \otimes_{\Z} \R$ and similarly for $Q_\R$ and $\cQ_\R$.  A Weyl chamber in $\cQ_\R$  is a connected component of the complement of the root hyperplanes $\la  \alpha, v \ra= 0$ indexed by $\alpha \in \Phi$. The Weyl group $W_0$ acts transitively on the Weyl chambers (\cite[VI, 1.5]{Bki-LA}).  

We consider the discrete subgroup 
$$\cLa= \{v \in \cQ_\R: \la \alpha, v \ra \in \Z \text{ for all } \alpha \in \Phi \} \subset \cQ_\R\ .$$ 
The dual basis  $(\check\omega_\alpha)_{\alpha\in \Pi}$  of  $\Pi$ is a basis for $\cLa$ (\cite[VI, 1.9 and 1.10]{Bki-LA}). We define
\begin{align}
\cLa^+ 
\notag &= \{\l \in \cLa: \la  \alpha, \l \ra \ge 0 \text{ for any } \alpha \in \Phi^+ \} \\
&= \{ \sum_{\alpha \in\Pi} a_\alpha \comega_\alpha: a_\alpha \in \Z_{\ge 0} \} \label{L+}.
\end{align}
and similarly for $\cX^+$.

\subsection{The length function}\label{sec:length}

Define the set of affine roots by  ${\Phi_\aff}=\Phi\times \mathbb Z={\Phi_\aff^+}\coprod {\Phi_\aff^-}$ where
$${\Phi_\aff^+}:=\{(\alpha , r)\vert\: \alpha \in\Phi, \,r>0\}\cup\{(\alpha ,0),\, \alpha \in\Phi^+\}, \quad{\Phi_\aff^-}:=\{(\alpha , r)\vert\: \alpha \in\Phi, \,r<0\}\cup\{(\alpha ,0),\, \alpha \in\Phi^-\}.$$
There is  a partial order on $\Phi$ given by $\alpha \preceq \beta$ if and only if $\beta -\alpha $ is a linear combination with (integral) nonnegative coefficients of elements in $\Pi$. Denote by $\Pi _m$ the set of roots that are minimal elements  for $\preceq$. The set of simple affine roots is  $\Pi _\aff:=\{(\alpha , 0),\: \alpha \in\Pi \}\cup\{(\alpha ,1),\, \alpha \in\Pi _m\}$. 
The extended affine Weyl group is $W = W_0 \ltimes \cX$. An element  $w_0 \cx \in W_0 \ltimes \cX$ acts on $\Phi_\aff$ by $w_0 \cx: (\alpha , r)\mapsto ({w_0}(\alpha), r- \la  \alpha, \cx \ra)$.

The length $\ell$ on the Coxeter system $({W_0}, S_0)$ extends to $W $ so that for any $A \in \Pi _\aff$ and $w \in W$
\begin{equation}\label{add}
\ell(w s_A)= 
   \begin{cases}
       \ell(w)+1 & \textrm{ if }w (A)\in {\Phi_\aff^+},\\  \ell(w)-1 & \textrm{ if }w (A)\in {\Phi_\aff^-}.
    \end{cases}
\end{equation}
where $s_A$ is the  affine reflection associated to $A$. The restriction of $\ell$ to $\cX \subset W$ has the following description
$$\ell(\cx)=\sum_{\alpha \in \Phi^+}\vert \la  \alpha, \cx \ra \vert$$
for any $\cx \in \cX$ (see \cite[\S I.10]{IM}). Thus one can extend $\ell$ to a piecewise linear function
\begin{equation}\label{eq:ell}
\ell: \cX_\R \longrightarrow \R_{\ge 0}, \quad \cx \longmapsto \sum_{\alpha\in \Phi^+}\vert \la \alpha, \cx \ra \vert \ .
\end{equation}
This function on $\cX_\R$ is strictly convex in the sense that it satisfies 
\begin{equation}
\ell(\cx+\cx')\leq \ell(\cx)+\ell(\cx')\label{trigX}\textrm{ for all $\cx,\cx'\in \cX_\R $}
\end{equation} 
with equality if and only if  $\la \alpha, \cx \ra$ and $\la \alpha, \cx' \ra$ have the same sign for all $\alpha \in \Phi^+$.
Notice also that $\ell$ is invariant under the action of $W_0$ on $\cX$.

\subsection{\label{cones}Cones} 

Recall that 
$$\cX^+:=\{\cx \in \cX \vert \: \la \alpha, \cx \ra \ge 0 \:\:\forall \alpha \in \Phi^+\}.$$ 
We will be interested in the following two conic monoids
\begin{align}
\label{def:coneX} \cC(\cX) & :=\{(\cx, k) \in \cX \oplus \Z: \: k \ge \ell(\cx)\}  \\
\notag \cC(\cX^+) & := \{(\cx, k) \in  \cC(\cX) : \: \cx \in  \cX^+\}  \subset \cC(\cX).
\end{align}
Since $\ell$ is invariant under $W_0$, the action of $W_0$ on $\cX$ lifts to an action on $\cC(\cX)$. 

When working with the corresponding semigroup algebras over $\r$ we will use multiplicative notation $e^\cx$ to denote the image of $\cx \in \cX$ in $\r[\cX]$. 
The semigroup algebra $\r[\cC(\cX)]$ is an $\r[\q]$-algebra via the map 
\begin{equation*}
\r[\q] \rightarrow \r[\cC(\cX)], \quad \q \mapsto e^{(0,1)}
\end{equation*}
Note that $\r[\cC(\cX)]$ is naturally a $\Z_{\ge 0}$-graded algebra with $e^{(\cx,k)}$ having degree $k$. In particular, $\q$ in degree one.  

Clearly $\r[\cC(\cX^+)] \subset \r[\cC(\cX)]$ is an $\r[\q]$-subalgebra. Since $\ell$ is additive on $\cX^+$ we have an isomorphism of $\r[\q]$-algebras
\begin{align}
\label{eq:iso+} \r[\cC(\cX^+)] & \longrightarrow \r[\q][\cX^+] \\
\notag e^{(\cx, k)} & \longmapsto \q^{k-\ell(\cx)} e^\cx
\end{align}

On the other hand, the length function $\ell$ on $\cX$ induces a filtration on $\r[\cX]$ 
\begin{equation}\label{fil}
F^i \r[\cX] :={\rm Span}_\r \big\{ e^\cx \in \r[\cX] :\: \ell(\cx) \leq i \big\}.
\end{equation}
We denote the associated Rees algebra 
$$\Rees \r[\cX] = \bigoplus_{i \ge 0} t^i F^i \r[\cX]$$
where $t$ is the Rees parameter. The following result is an easy exercise. 

\begin{lemma}\label{lem:rees}
There exists an isomorphism of $\r$-algebras 
\begin{align*}
\r[\cC(\cX)] & \to \Rees \r[\cX] \\
e^{(\cx,k)} & \mapsto t^k e^\cx .
\end{align*}
\end{lemma}

\begin{remark}\label{rem:filtered}
On several occasions we will make use of the elementary fact that if $\phi$ is a map of $\N$-filtered $\r$-modules and $\gr \phi$ is an isomorphism then so are $\phi$ and  the corresponding map of Rees algebras.
\end{remark}

\subsection{Generic Hecke algebras} \label{subsec:Hecke}

The  (extended) generic affine Hecke algebra $\cH_\q$ (cf. \cite[Sect. 3.2]{Lu1}) is the $\r[\q]$-algebra which is free as an $\r[\q]$-module with basis $\{  T_w\}_{w \in W}$ and subject to relations 
\begin{align*} 
 T_v   T_w =   T_{vw} & \quad \text{ if } \ell(vw)=\ell(v)+\ell(w)  \\
( T_s - \q^2)(  T_s + 1)=0 & \quad \text{ for  } s \in \Saff.
\end{align*} 

Following \cite[3.3 (a)]{Lu1}, we define in $\theta_\cx \in \cH_{\q^{\pm 1}}$  for $\cx \in \cX$  as follows. Write $\cx = \cx_1-\cx_2$ with $\cx_1, \cx_2 \in \cX^+$ and let 
$$\theta_\cx:= \q^{{\ell(\cx_2)-\ell(\cx_1)}}  T_{\cx_1} T_{\cx_2}^{-1}\ .$$
By \cite[3.3 (b) and Lemma 3.4]{Lu1},  the map
\begin{align}\label{theta}
\r[\q^{\pm 1}][\cX] & \longrightarrow  \cH_\q \otimes_{\r[\q]} \r[\q^{\pm 1}]\cr 
\cx & \longmapsto \theta_\cx
\end{align}
is an injective homomorphism of $\r[\q^{\pm 1}]$-algebras. We denote by $\cA_{\q^{\pm 1}}$ its image. It is endowed with a 
 $\r[\q^{\pm 1}]$-action of $W_0$ via the action of $W_0$ on $\cX$. The center $\cZ_{\q^{\pm 1}} \subset \cH_{\q^{\pm 1}}$ is the algebra of $W_0$-invariants $\cA_{\q^{\pm 1}}^{W_0}$ (\cite[Prop. 3.11]{Lu1}).
 
To define a generic version of $\cA_{\q^{\pm 1/2}}$ consider the renormalized elements 
\begin{equation}\label{Ex}
E_\cx:= \q^{{\ell(\cx)}}\theta_\cx\ .
\end{equation} 
Note that $E_\cx=T_\cx$ if $\cx \in \cX^+$. These elements were first introduced in \cite[Prop. 7]{Vig1} in the Iwahori-Hecke algebra attached to a $p$-adic reductive group and then in \cite[Thm. 2.7, Cor 2.8 and Example 5.30]{VigI} for the generic Iwahori-Hecke algebra. Following \emph{loc. cit} they satisfy the following properties:
\begin{enumerate}
\item[(i)] $E_\cx \in \cH_\q$ for all $\cx \in \cX$,
\item[(ii)]  $E_\cx \in  T_\cx + \sum_{w \in W, \ell(w)<\ell(\cx)} \r[\q] T_w$,
\item[(iii)]  $E_\cx E_{\cx'}= \q^{{\ell(\cx)+\ell(\cx')-\ell(\cx+\cx')}} E_{\cx+\cx'} \in \cH_\q$ for all $\cx,\cx'\in \cX$. 
 \end{enumerate}
 
\begin{lemma}\label{lem:A}
The following is an injective map of $\r[\q]$-algebras
\begin{align}\label{E}
\r[\cC(\cX)] & \longrightarrow  \cH_{\q} \cr  
e^{(\cx, k)} & \longmapsto \q^{{k-\ell(\cx)}}E_\cx \ .
\end{align} 
\end{lemma}
\begin{proof}
The map is a homomorphism of algebras because of properties (i) and (iii) above. It is injective since \eqref{theta} is injective.
\end{proof} 
We denote by $\cA_\q$ the image of \eqref{E}.

\section{The center of $\cH_\q$}

The main goal in this section is to prove the following result. 

\begin{theorem}\label{thm:main}
There exists an isomorphism of $\r[\q]$-algebras 
\begin{equation}\label{eq:main1}
\r[\cC(\cX)]^{W_0} \cong \r[\q][\cX^+].
\end{equation}
\end{theorem}

By \cite[Thm. 1.2]{VigII} (cf. \cite[Thm. 4]{VigII}) the center $\cZ_\q \subseteq \cH_\q$ is given by $\cZ_\q = (\cA_\q)^{W_0}$.  Using Lemma \ref{lem:A} we get the following.

\begin{corollary}\label{coro:main}
There exists  isomorphisms of $\r[\q]$-algebras $\cZ_\q \cong  \r[\q][\cX^+] \cong \r[\q][\cX]^{W_0}$.
\end{corollary}

\subsection{Orbits}\label{sec:orbits}

We will be making use of the projection 
\begin{equation}\label{eq:proj}
\begin{array}{cccl} p: & \cX_\R & \longrightarrow & \cX_\R \cr 
& \cx & \longmapsto & \frac{1}{\vert W_0 \vert} \sum_{w \in W_0} w(\cx) 
\end{array}
\end{equation}
We denote $q(\cx) := \cx - p(\cx)$. 

\begin{lemma}\label{lem:proj}
For any $\cx \in \cX$ we have $q(\cx) \in \cLa$. 
\end{lemma}
\begin{proof}
By induction on on $\ell(w)$ it follows that $\cx-w(\cx)\in \cQ$ for any $w \in W_0$. Therefore $\cx-p(\cx) \in \cQ_\R$. Furthermore, since $p(\cx)$ is $W_0$-invariant, we have $\la \alpha, p(\cx) \ra = 0$ and so $\la \alpha, \cx-p(\cx) \ra = \la \alpha, \cx \ra \in \Z$ for all $\alpha \in \Phi$. Hence $\cx-p(\cx) \in \cLa$.
\end{proof}

\begin{lemma}  \label{lem:orbit} 
Each $W_0$-orbit in $\cX$ intersects $\cX^+$ in exactly one point.
\end{lemma}
\begin{proof}
We first check that each $W_0$-orbit in $\cX$ intersects $\cX^+$ in at most one point.  Let $\cx \in \cX^+$ and $w \in W_0$ such that $w(\cx) \in \cX^+$. We may  take $w$ to be of minimal length in the left coset $w \, {\rm Stab}_{W_0}(\cx)$.  Assume $\ell(w) \geq 1$ and let $\alpha\in \Pi$  such that $\ell(s_{\calpha} w)=\ell(w)-1$. This means that $w^{-1}(\alpha)\in \Phi^-$ so $\la w^{-1}(\alpha), \cx \ra \le 0$. But  $\la  w^{-1}(\alpha) , \cx \ra = \la \alpha , w(\cx) \ra \ge 0$. This means $\la \alpha , w(\cx) \ra=0$ and hence $s_{\calpha}(w(\cx))= w(\cx)$ and  $s_{\calpha} w \in w \,{\rm Stab}_{W_0}(\cx)$. This contradicts the minimality of $w$. 

It remains to show that each $W_0$-orbit intersects $\cX^+$ at least once.  By Lemma \ref{lem:proj} we know $\cx-p(\cx) \in \cLa$ so choose $w \in W_0$ such that $w(\cx-p(\cx))= w(\cx)-p(\cx) \in \cLa^+$. For $\alpha \in \Phi^+$  we have $\la \alpha, w(\cx) \ra =\la \alpha, w(\cx)-p(\cx) \ra \in \Z_{\geq 0}$. So $w(\cx) \in \cX^+$.
\end{proof}

For $\cx \in \cX$ denote by $\cO(\cx)$ its $W_0$-orbit. By Lemma \ref{lem:orbit} there is a bijection 
\begin{align*}
\cX^+ & \longrightarrow \{ \text{ $W_0$-orbits in $\cX$ }  \} \\ 
\cx & \longmapsto \cO(\cx) \ .
\end{align*}
For $\cx \in \cX$ we denote
\begin{align}
\label{zl1} z_{\cO(\cx)} & :=  \sum_{\cx' \in \cO(\cx)} e^{\cx'} \in \r[\cX]  \ .
\end{align} 
In particular, $\{z_{\cO(\cx)}\}_{\cx \in \cX^+}$ is an $\r$-basis for $\r[\cX]^{W_0}$.

\begin{lemma} \label{lem:orbit2}
If $\cl, \cl' \in \cLa^+$, then the map
 \begin{align*}
 \{(\cmu,\cmu')\in  \cO(\l)\times  \cO(\l')\,\vert\, \cmu\sim \cmu'\} & \to \cO(\cl+\cl') \\
 (\cmu,\cmu') & \mapsto \cmu+\cmu'
 \end{align*}
 is a bijection, where we write $\cmu\sim \cmu'$ if they belong to a common Weyl chamber. 
\end{lemma}
\begin{proof}
The map is well defined by Lemma \ref{lem:orbit} and is clearly surjective. We compare cardinalities to conclude that it is bijective. More precisely, since ${\Stab}(\cl) \subset W_0$ is generated by all $s_{\calpha}$ with $\la \alpha, \cl\ra =0$, it follows that $\Stab(\cl+\cl') = \Stab(\cl) \cap \Stab(\cl')$ and hence
$$|\cO(\cl+\cl')| = |W_0|/|\Stab(\cl) \cap \Stab(\cl')|.$$ 
On the other hand, the cardinality of $\cO(\cl)$ is  $|W_0|/|\Stab(\cl)|$. Moreover, having fixed $\cmu\in \cO(\cl)$, we have 
$\vert\{\cmu'\in \cO(\cl')\,\vert\, \cmu'\sim \cmu\}\vert=
\vert\{\cmu'\in \cO(\cl')\,\vert\, \cmu'\sim \cl\}\vert=
 \vert \Stab(\cl)\vert/\vert \Stab(\cl) \cap \Stab(\cl')\vert$ since the distinct Weyl chambers containing $\cl$ are all the $u\cLa^+$ for $u\in \Stab(\cl)$. The result follows.
\end{proof}

\subsection{Case $\cX = \cLa$}

We first prove Theorem \ref{thm:main} in the adjoint case when $\cX=\cLa$. 

\begin{lemma}\label{lem:fact}
For $\cl, \cl' \in \cLa^+$, we have
$$z_{\cO(\cl)} z_{\cO(\cl')} - z_{\cO(\cl+\cl')} \in F^{\ell(\cl+\cl')-1} \r[\cLa]^{W_0}$$
where $z_{\cO(\cl)} \in \r[\cLa]^{W_0}$ is defined by \eqref{zl1}.
\end{lemma}
 \begin{proof}
Recall that for $\cmu \in \cO(\cl)$ and $\cmu' \in \cO(\cl')$ we write $\cmu \sim \cmu'$ if they lie in a common Weyl chamber, namely a common $W_0$-conjugate of $\cLa^+$. We then have 
\begin{align*} 
z_{\cO(\cl)}z_{\cO(\cl')}
&= \left(\sum_{\mu \in \cO(\cl)}e^{\cmu} \right) \left( \sum_{\cmu' \in \cO(\l')} e^{\cmu'} \right) 
= \sum_{\cmu \sim \cmu'}  e^{\cmu+\cmu'} + \sum_{\cmu \not \sim \cmu'} e^{\cmu+\cmu'}\ .
\end{align*}
Lemma \ref{lem:orbit2} shows that
$$ z_{\cO(\cl)}z_{\cO(\cl')} = z_{\cO(\cl+\cl')} + \sum_{\cmu \not \sim \cmu'} e^{\cmu +\cmu'}.$$
Since $\ell(\cmu+\cmu')=\ell(\cmu)+\ell(\cmu')$ if $\cmu \sim \cmu'$ and $\ell(\cmu+\cmu')<\ell(\cmu)+\ell(\cmu')$ otherwise, we see that the right hand terms above belong to $F^{\ell(\cl+\cl')-1} \r[\cLa]^{W_0}$.
\end{proof}

\begin{proposition} \label{prop:lam} 
There exists an isomorphism of $\N$-filtered $\r$-algebras 
\begin{align} \label{map:poly}
\phi_{\cLa}: \r[\cLa^+] & \longrightarrow  \r[\cLa]^{W_0} \cr 
e^{\comega_\alpha} & \longmapsto z_{\cO(\comega_\alpha)} \quad\quad\quad \text{  for $\alpha \in \Pi$.} 
\end{align}
\end{proposition} 
\begin{proof}
Since $\r[\cLa^+]$ is freely generated by  $\{e^{\comega_\alpha}\}_{\alpha\in \Pi}$ (see  \eqref{L+}) the map \eqref{map:poly} is well defined. 


On the other hand, for any $\cla=\sum_{\alpha\in \Pi} n_\alpha \comega_\alpha \in \cLa^+$ with $n_\alpha\in \Z_+$, we have 
\begin{equation}\label{eq:prod}
\phi_{\cLa}(e^\cla) - z_{\cO(\cla)} = \prod_{\alpha} (z_{\cO(\comega_\alpha)})^{n_\alpha} - z_{\cO(\cla)} \in F^{\ell(\cla)-1} \r[\cLa]^{W_0}
\end{equation}
where the rightmost inclusion is by Lemma \ref{lem:fact}. It follows that $\phi_{\cLa}$ preserves the filtration and that its associated graded is
\begin{align}
\notag \gr \phi_\cLa: \gr \r[\cLa^+] & \to \gr \r[\cLa]^{W_0} \\
\label{eq:phi} e^\cla & \mapsto  z_{\cO(\cla)} \ .
\end{align}
By Lemma \ref{lem:orbit}, the set $\{z_{\cO(\cla)}\}_{\cla \in \cLa^+}$ is a $\r$-basis for $\r[\cLa]^{W_0}$ which implies that $\gr \phi_\cLa$ is a bijection. It follows that $\phi_{\cLa}$ is an isomorphism (cf. Remark \ref{rem:filtered}). 
\end{proof} 

\begin{lemma}\label{lem:lam}
The composition of $\phi_{\cLa}$ with the projection to $\r[\cLa/\cQ]$ is given by 
\begin{align}
\notag \r[\cLa^+] & \rightarrow \r[\cLa/\cQ] \\
\label{eq:lam} e^{\cla} & \mapsto  \prod_{\alpha \in \Pi} |\cO(\comega_\alpha)|^{n_\alpha} [e^{\cla}]
\end{align}
for any $\cla = \sum_{\alpha\in \Pi} n_\alpha \comega_\alpha \in \cLa^+$. 
\end{lemma}
\begin{proof}
The action of $W_0$ on $\cLa/\cQ$ is trivial (this is clear since $s_\alpha \cla = \cla - \la \cla, \alpha \ra \calpha$). It follows that $[w \cdot \comega_\alpha] = [\comega_\alpha] \in \cLa/\cQ$ for any $w \in W_0$ and hence $e^{\comega_\alpha} \mapsto [z_{\cO(\comega_\alpha)}] = |\cO(\comega_\alpha)| [e^{\comega_\alpha}]$.  The result follows since the map is a homomorphism.
\end{proof}

\subsection{Proof of Theorem \ref{thm:main}}\label{sec:proof}

Consider the injective map 
\begin{align}\label{embed'}
 \: \:\cX  &\longrightarrow p(\cX) \oplus \cLa = \cY \\
 \notag \cx & \longmapsto (p(\cx), q(\cx)) 
\end{align} which is well defined because $q(\cx) \in \cLa$ by Lemma \ref{lem:proj},  and which  maps $\cx\in \cQ$  onto $(0,\cx)$ because each coroot is $W_0$-conjugate to its opposite. Consider the following diagram 
\begin{equation}\label{diag1}
\xymatrix{
\r[\cX^+] \ar@{^{(}->}[d]  \ar[drrrr] \ar@{-->}[rrrr]^{\phi_\cX} &&&& \r[\cX]^{W_0} \ar@{^{(}->}[d]  \\
\r[\cY^+] \ar[rrrr]^{\phi_{\cY}} &&&& \r[\cY]^{W_0}  } 
\end{equation} 
where  the vertical arrows  are induced by \eqref{embed'},
$$\phi_{\cY} = \id_{\r[p(\cX)]} \otimes \phi_{\cLa}:  \r[\cY^+] \longrightarrow  \r[\cY]^{W_0} = \r[p(\cX)] \otimes_\r \r[\cLa]^{W_0}$$
and the diagonal is the composition of the left and bottom arrows. 

First we explain why the diagonal map factors through some dotted map $\phi_\cX$. Equivalently, we show that the composition 
$$\r[\cX^+] \hookrightarrow \r[\cY^+] \xrightarrow{\phi_{\cY}} \r[\cY]^{W_0} \hookrightarrow \r[\cY] \to  
 \r[\cY]/\r[\cQ]$$
has image in $\r[\cX]/\r[\cQ]$.  This is because by Lemma \ref{lem:lam},  the composition above  maps  $\cx\in \cX^+$ to a multiple of $e^{p(x)}\otimes [e^{q(x)}]\in \r[p(\cX)\oplus \cLa/\cQ]\cong \r[\cY]/\r[\cQ]$ and this element lies in the subspace $\r[\cX]/\r[\cQ]$ of $\r[\cY]/\r[\cQ]$.

Finally, note that there is also a well defined length function $\ell$ on $\cY$ which restricts to zero on $p(\cX) \subset \cY$. The restriction of this length function from $\cY$ to $\cX$ is the length function on $\cX$. This implies that the diagonal map in \eqref{diag1} preserves the filtration induced by $\ell$. Using that $\gr \phi_\cLa$ is given by \eqref{eq:phi}, the associated graded of this diagonal map is
\begin{align*}
\gr \r[\cX^+] & \to \gr \r[\cY]^{W_0} \\
e^{\cx} & \mapsto e^{p(\cx)}\otimes z_{\cO(q(\cx))} \ .
\end{align*}
By Lemma \ref{lem:orbit} this map is injective  with image $\gr \r[\cX]^{W_0}$. It follows that 
$$\gr \phi_\cX: \gr \r[\cX^+] \to \gr \r[\cX]^{W_0}$$
is an isomorphism and thus so is $\phi_{\cX}$ (cf. Remark \ref{rem:filtered}). Applying the Rees construction induces an isomorphism
$$\Rees (\phi_\cX) : \r[\q][\cX^+] \to \r[\cC(\cX)]^{W_0}$$
which yields \eqref{eq:main1}. 

\begin{remark}
If we quotient \eqref{eq:main1} from Theorem \ref{thm:main} by $\q=1$ then we get an isomorphism of $\r$-algebras $\r[\cX]^{W_0} \cong \r[\cX^+]$.
\end{remark}

\section{Structure of $\cH_\q$ over its center}\label{sec:frob}

We assume in this section that $\r$ is a Noetherian ring. It is proved in \cite[Thm. 1.3]{VigII} that $\cH_\q$ is finitely generated as a module over its center $\cZ_\q$ which is a finitely generated $\r[\q]$-algebra.

 As a consequence of \cite[Cor. 5.7]{CO1} and Corollary \ref{coro:main} (see also \cite[Cor. 3.13]{CO1}) we have the following.

\begin{corollary}\label{cor:rigid} Assume that $\r$ is a regular, finitely generated $k$-algebra where $k$ is a field.
If $X/Q$ is free then $\cH_\q$ is  finitely generated, projective over $\cZ_\q$ and the $\cZ_\q$-rigid dualizing complex of $\cH_\q$ is
${\bfR} _{\cH_\q/\cZ_\q}=(\upiota) \cH_\q$.
\end{corollary}

Here $\upiota$ is an involution on $\cH_\q$ which can be described as follows. First recall  that $\cH_\q$ is equipped with a $\r[\q]$-algebra involution $\iota$ (see \cite[Prop. 2.3]{CO1}) satisfying
\begin{equation}\label{f:invo}
 \iota(T_\omega)= T_\omega \text{ if $\ell(\omega)=0$}\quad\text{ and }\quad \iota(T_s-1)=-(T_s+\q^2) \text{ for $s\in \Saff$ } .
 \end{equation}
 Then $\upiota:= \iota \circ j$ where $j$ is an involution of $\cH_\q$ defined on its basis by $j:T_{w}\mapsto \epsilon(w) T_{w}$ for $w\in W$ where $\epsilon$ is a certain orientation character (cf. \cite[Remark 2.13]{CO1}). To be precise, the extended Weyl group $W$ is the semi-direct product of the normal subgroup $W_\aff$ generated by $\{s_A\}_{A \in \Pi_\aff}$ and the subgroup $\Omega$ of all elements of length zero. Then $\epsilon:W  \rightarrow W/\Waff\cong \Omega \rightarrow \{\pm 1\}$ where the second map is the signature of $\Omega$ acting on $\{s_A\}_{A \in \Pi_\aff}$. 


Finally, as a consequence of \cite[Cor. 4.10]{CO1} and Corollary \ref{coro:main}, we obtain the following.
\begin{corollary}\label{cor:frob}
If $X/Q$ is free and $\r = \k$ then $\cH_\q$ is a free Frobenius extension over its center with Nakayama automorphism $\upiota$. 
\end{corollary}

\section{Toric geometry}\label{sec:toric}

In this section we take $\r = \Q$ and restrict ourselves to the case when the root datum corresponds to a group of adjoint type, meaning that $X = Q$ and $\cX = \cLa$. We denote by $\Sigma$ the fan in $\cX_\R$ associated to the Weyl chambers. Namely, for each subset $\Delta \subset \Phi$ we consider the cones $\{v \in \cX_\R: \la v, \alpha \ra \ge 0, \forall \alpha \in \Delta\}$. We let $X_\Sigma$ denote the corresponding toric variety. This is a smooth, proper toric variety that has been studied in various contexts in the literature. 

For example, when the root system corresponds to a semi-simple Lie algebra, it plays a central role in the description of the minimal wonderful compactification \cite{CP}. On the other hand, by work of Losev and Manin \cite{LM}, when the root system is of type $A_{n-1}$ the variety $X_\Sigma$ can be interpreted as the fine moduli space $\barL_n$ of stable, $n$-pointed chains of projective lines. Moreover, there exists a natural surjective birational morphism $\barM_{0,n+2} \to \barL_n$ from the Grothendieck-Knudsen moduli space and subsequently one can interpret $\barL_n$ as a moduli space of weighted pointed stable curves as constructed by Hassett \cite{Ha} (where the weight is $(1,1,\frac1n,\dots\frac1n)$). This interpretation of $X_\Sigma$ as a moduli space was further studied by Batyrev and Blume \cite{BB1,BB2} where, among other things, they give analogous moduli theoretic descriptions of $X_\Sigma$ in other types. 

\subsection{Equivariant cohomology}

The action of the Weyl group $W_0$ on the cohomology of $X_\Sigma$ was studied by Procesi \cite{Pr}, Dolgachev and Lunts \cite{DL} and Stembridge \cite{Stem}.  The cohomology of $X_\Sigma$ can be described in terms of the Stanley-Reisner ring associated to $\Sigma$. We review this description and use it to relate this cohomology to the ring $\cA_0 = \cA_\q/ \q\cA_\q$ (Proposition \ref{prop:coh1}). Note that $\cA_0$ is the associated graded with respect to the filtration on $\Q[\cX]$ induced by $\ell$. 
Via the identification $e^x\mapsto E_x \bmod \q \cA_\q$ (see \eqref{E}), we see $\{e^x\}_{x\in\cX}$  as a basis for $\cA_0$.

Let $G(\Sigma) = \{v_1, \dots, v_n\}$ denote the generators of the rays (one dimensional cones) of $\Sigma$. A subset $\{v_{i_1}, \dots, v_{i_p}\} \subset G(\Sigma)$ is called primitive if it is not the set of generators of a $p$-dimensional cone in $\Sigma$ but every $k$-subset with $k < p$ does generate a $k$-dimensional cone in $\Sigma$. 

A continuous function $\varphi: \cX_\R \to \R$ is called $\Sigma$-piecewise linear if its restriction to every cone is linear. It is easy to see that we have a natural isomorphism 
$$PL(\Sigma) \to \R^n \ \ \ \varphi \mapsto (\varphi(v_1), \dots, \varphi(v_n))$$
where $PL(\Sigma)$ is the space of $\Sigma$-piecewise linear functions. A function $\varphi \in PL(\Sigma)$ is (upper) convex if 
$$\varphi(x_1 + x_2) \le \varphi(x_1) + \varphi(x_2)$$
for $x_1,x_2 \in \cX_\R$. It is called strictly convex if the inequality above is strict whenever $x_1,x_2$ do not belong to a common cone. 

\begin{remark}\label{rem:convex}
Following the discussion in Section \ref{sec:length} the length function $\ell: \cX_\R \to \R$ is piecewise linear, strictly convex. 
\end{remark}

\begin{remark}
There is a natural inclusion $X_\R \to PL(\Sigma)$ corresponding to globally linear functions. The quotient $PL(\Sigma)/X_\R$ can be identified with $H^2(X_\Sigma,\R)$. Under this identification the image of convex $\Sigma$-piecewise linear functions is the closed K\"ahler cone of $X_\Sigma$.  This motivates the relevance of convexity geometrically.  
\end{remark}

\begin{lemma}\label{lem:convex1}
If $\{x_1,\dots,x_k\} \subset \cX$ then, inside $\cA_0$, we have
$$e^{x_1} \dots e^{x_k} = 
\begin{cases}
e^{x_1+\dots+x_k} & \text{ if $x_1,\dots,x_k$ belong to a cone in $\Sigma$ } \\
0 & \text{ otherwise }
\end{cases}$$
\end{lemma}
\begin{proof}
This follows from the fact that $\ell$ is $\Sigma$-piecewise linear and strictly convex (Remark \ref{rem:convex}) and that $\cA_0$ is the associated graded with respect to the filtration on $\Q[\cX]$ induced by $\ell$. See also (iii) before Lemma \ref{lem:A}).
\end{proof}

Let $\Q[\uz] = \Q[z_1,\dots,z_n]$ and denote by $\SR(\Sigma) \subset \Q[\uz]$ the ideal generated by all monomials $z_{i_1} \dots z_{i_p}$ where $\{v_{i_1},\dots,v_{i_p}\}$ is primitive.  This is the Stanley-Reisner ideal. Then the equivariant cohomology ring $H^*_T(X_\Sigma,\Q)$ can be identified with $\Q[\uz]/\SR(\Sigma)$ (see \cite{BCP}). 

\begin{proposition}\label{prop:coh1}
There exists an isomorphism $H^*_T(X_\Sigma,\Q) \cong \cA_0$.
\end{proposition}
\begin{proof}
Consider the map of algebras 
\begin{align*}
\tau: \Q[\uz] & \to \cA_0  \\
z_i & \mapsto e^{v_i}
\end{align*}
By Lemma \ref{lem:convex1} this map factors through $\Q[\uz]/\SR(\Sigma)$. 
Now take some $x \in \cX$ and consider the minimal cone $\sigma \in \Sigma$ containing it. Write $x = \sum_{j=1}^k a_j v_{i_j}$ where $a_j \in \N$ and $\{v_{i_1}, \dots, v_{i_k}\}$ are the generators of $\sigma$. Consider the map of vector spaces 
\begin{align*}
\tau': \cA_0 & \to \Q[\uz]/\SR(\Sigma) \\
e^x & \mapsto z_{i_1}^{a_1} \dots z_{i_k}^{a_k}
\end{align*}
Notice that 
$$\tau \circ \tau'(e^x) = \tau(z_{i_1}^{a_1} \dots z_{i_k}^{a_k}) = e^{a_1 v_{i_1}} \dots e^{a_k v_{i_k}} = e^{\sum_j a_j v_{i_j}} = e^x$$
where the second equality follows from Lemma \ref{lem:convex1}. This proves that $\tau$ is onto. On the other hand, $\Q[\uz]/\SR(\Sigma)$ has a basis given by monomial $z_{i_1}^{a_1} \dots z_{i_k}^{a_k}$ where $v_{i_1}, \dots, v_{i_k}$ belong to a common cone. It follows that $\tau'$ is also onto which means $\tau$ is injective. 
\end{proof}

\subsection{Quantum cohomology}

We now explain how to extend Proposition \ref{prop:coh1} to an isomorphism involving $\cA_\q$. To do this most naturally first recall that, by Lemma \ref{lem:rees}, $\cA_\q = \r[\cC(\cX)]$ is the Rees of $\r[\cX]$ for the filtration induced by the length function $\ell$. The key property used in this identification is that $\ell$ is a $\Sigma$-piecewise linear, convex (the triangle inequality) function. Thus, given an arbitrary convex function $\varphi \in PL(\Sigma)$ we can consider the analogous $\r$-algebra $\cA^\varphi_\q = \Rees^\varphi \r[\cX]$. In this notation $\cA_\q = \cA_\q^\ell$. One can also define the cone 
$$\cC^\varphi(\cX) = \{ (\cx,k) \in \cX \oplus \Z: k \ge \varphi(\cx) \}$$
so that $\cA_\q^\varphi \cong \Q[\cC^\varphi(\cX)]$ where, as before, $\q$ corresponds to $e^{(0,1)}$. 

\begin{lemma}\label{lem:convex2}
If $\{w_1,\dots,w_k\} \subset \cX$ then, inside $A^\varphi_\q$, we have
$$e^{w_1} \dots e^{w_k} = \q^{\sum_j \varphi(w_j) - \varphi(\sum_j w_j)} e^{w_1+\dots+w_k}$$
\end{lemma}
\begin{proof}
The proof is the same as that of Lemma \ref{lem:convex1} using that $\cA^\varphi_\q$ is the Rees algebra with respect to the filtration induced by $\varphi$. 
\end{proof}

To define the quantum cohomology $QH^*_{T,\varphi}(X_\Sigma,\Q)$ let $\Q[\uz,\q] = \Q[z_1,\dots,z_n,\q]$ and denote by $\SR_\q^\varphi(\Sigma) \subset \Q[\uz,\q]$ the quantum Stanley-Reisner ideal generated by 
\begin{equation}\label{eq:RSq}
\q^{\sum_{i=1}^n b_i \varphi(v_i) - \sum_{j=1}^k a_j \varphi(v_{i_j})} \prod_{i=1}^k z_{i_j}^{a_j} - \prod_{i=1}^n z_i^{b_i}
\end{equation}
over all possible $a_j,b_i \in \Z_{\ge 0}$ such that 
$\sum_{j=1}^k a_j v_{i_j} = \sum_{j=1}^n b_i v_i$
and $\{v_{i_1}, \dots, v_{i_k}\}$ belong to the same cone. The ring $QH^*_{T,\varphi}(X_\Sigma,\Q)$ is then the quotient $\Q[\uz,\q]/\SR_\q^\varphi(\Sigma)$. As the following Lemma explains, this ring is a $\q$-integral (equivariant) version of the quantum cohomology ring defined in \cite[Def. 5.1]{Ba}.

\begin{lemma}\label{lem:batyrev}
Specializing $\q^{\cdot} \mapsto \exp(-(\cdot))$ in $\SR^\varphi_\q(\Sigma)$ recovers the ideal $Q_\alpha(\Sigma)$ from \cite[Def. 5.1]{Ba}.
\end{lemma} 
\begin{proof}
The ideal $\SR^\varphi_{\exp(-)}(\Sigma)$ is clearly contained in the ideal  $Q_\alpha(\Sigma)$ defined in \cite[Defn. 5.1]{Ba}. On the other hand, by \cite[Thm. 9.5]{Ba}, the ideal $Q_\alpha(\Sigma)$ is equivalently generated by 
$$\exp{(-\sum_{i=1}^n b_i \varphi(v_i))} - \prod_{i=1}^n z_i^{b_i}$$
where $\sum_{i=1}^n b_i v_i = 0$ with $b_i \in \Z_{\ge 0}$. Such relations are special cases of those from \eqref{eq:RSq} where all $a_i$ are zero. Thus we also have $Q_\alpha(\Sigma) \subset \SR^\varphi_{\exp(-)}(\Sigma)$.
\end{proof}

\begin{proposition}\label{prop:coh2}
There exists an isomorphism $QH^*_{T,\varphi}(X_\Sigma,\Q) \cong A_\q^\varphi$.
\end{proposition}
\begin{proof}
The proof is similar to that of Proposition \ref{prop:coh1}. Consider the map of algebras
\begin{align*}
\tau_\q: \Q[\uz,\q] & \to A_\q^\varphi  \\
z_i & \mapsto \q^{\varphi(v_i)} e^{v_i}
\end{align*}
By Lemma \ref{lem:convex2} this map factors through $\Q[\uz,\q]/\SR^\varphi_\q(\Sigma)$. 

For $w \in \cX$ we write it as $w = \sum_{j=1}^k a_j v_{i_j}$ where $a_j \in \N$ and $\{v_{i_1}, \dots, v_{i_k}\}$ belong to the same cone. Consider the map of vector spaces 
\begin{align*}
\tau_\q': \cA_\q^\varphi & \to \Q[\uz]/\SR_\q^\varphi(\Sigma) \\
\q^k e^w & \mapsto \q^{k-\varphi(w)} z_{i_1}^{a_1} \dots z_{i_k}^{a_k}
\end{align*}
As before $\tau_\q \circ \tau'_\q$ is the identity map which implies that $\tau_\q$ is onto. On the other hand, using relations \eqref{eq:RSq}, $\Q[\uz]/\SR_\q^\varphi(\Sigma)$ has a basis given by monomial $\q^k z_{i_1}^{a_1} \dots z_{i_k}^{a_k}$ where $v_{i_1}, \dots, v_{i_k}$ belong to a common cone and $k \ge 0$ . It follows that $\tau_\q'$ is onto which means $\tau_\q$ is injective. 
\end{proof}

\begin{remark} 
It is noted in \cite{Ba} that in the limit $\varphi \to \infty$ the quantum cohomology ring recovers the usual cohomology ring. In light of Propositions \ref{prop:coh1} and \ref{prop:coh2} this corresponds to recovering $\cA_0$ from $\cA_\q$ be specializing $\q \mapsto 0$. 
\end{remark}

\end{document}